\documentclass{amsart}

\usepackage{amssymb}
\usepackage{amsmath}
\usepackage{amsthm}
\usepackage{amsfonts}

\usepackage{a4wide}
\usepackage{longtable}
\usepackage{multirow}

\newtheorem{lemma}{Lemma}[section]
\newtheorem{proposition}{Proposition}[section]
\newtheorem{theorem}{Theorem}[section]
\newtheorem{corollary}{Corollary}[theorem]

\theoremstyle{definition}
\newtheorem{definition}{Definition}[section]

\theoremstyle{remark}
\newtheorem{remark}{Remark}[section]
\newtheorem{example}{Example}[section]

\def\fbar{\overline{f}}

\begin{document}

\title{\(3\)-class field towers of exact length \(3\)}

\author{M. R. Bush}
\address{Mathematics Department\\Washington and Lee University\\Lexington\\Virginia 24450\\USA}
\email{bushm@wlu.edu}
\author{D. C. Mayer}
\address{Naglergasse 53\\8010 Graz\\Austria}
\email{algebraic.number.theory@algebra.at}
\urladdr{http://www.algebra.at}


\date{December 1, 2013}


\begin{abstract}
The $p$-group generation algorithm is used to verify that the Hilbert \(3\)-class field tower has length $3$
for certain imaginary quadratic fields \(K\) with \(3\)-class group \(\mathrm{Cl}_3(K) \cong [3,3]\).
Our results provide the first examples of finite \(p\)-class towers of length $> 2$ for an odd prime \(p\).
\end{abstract}

\maketitle


\section{Introduction}
\label{s:Intro}

In 1925, Schreier and Furtw\"angler \cite[\S\ 15.1.1, p. 218]{FRL} asked whether the ascending tower
\(K\le\mathrm{F}^1(K)\le\mathrm{F}^2(K)\le\ldots\) of successive Hilbert class fields of an algebraic number field \(K\)
can be infinite \cite[\S\ 11.3, p. 46]{Ha1}. In their famous 1964 paper \cite{GoSh}, Golod and Shafarevich gave an affirmative answer. They did this by proving that the tower of Hilbert \(p\)-class fields \(K\le\mathrm{F}_p^1(K)\le\mathrm{F}_p^2(K)\le\ldots\)
(which sits inside the tower of Hilbert class fields) is infinite if the base field \(K\) has sufficiently large \(p\)-class rank \(d_p(\mathrm{Cl}(K))\) where $p$ is some prime. They combined this result with a theorem warranting large \(p\)-class rank \(d_p(\mathrm{Cl}(K))\) whenever
sufficiently many primes ramify completely in \(K\) with exponents divisible by \(p\), and thus showed that the \(2\)-tower of a quadratic field
\(K=\mathbb{Q}(\sqrt{D})\) with highly composite radicand \(D\) is infinite. For example, taking \(D=-2\cdot 3\cdot 5\cdot 7\cdot 11\cdot 13=-30\,030\)
or \(D=2\cdot 3\cdot 5\cdot 7\cdot 11\cdot 13\cdot 17\cdot 19=9\,699\,690\) one obtains a \(2\)-tower of length \(\ell_2(K)=\infty\).

A year earlier, Shafarevich \cite{Sh} had proved that if $p$ is an odd prime then $r = d$ where
\(d=\dim_{\mathbb{F}_p}(\mathrm{H}_1(G,\mathbb{F}_p))\)
and
\(r=\dim_{\mathbb{F}_p}(\mathrm{H}_2(G,\mathbb{F}_p))\)
are the generator and relation ranks of the \(p\)-tower group
\(G=\mathrm{G}_p^\infty(K)=\mathrm{Gal}(\mathrm{F}_p^\infty(K)\vert K)\) and \(K\) is an imaginary quadratic field.
Together with the general condition \(r>\frac{1}{4}d^2\) for a finite \(p\)-tower group, this established the bound \(d<4\) for an imaginary quadratic field with finite $p$-tower, which was improved to \(d<3\) by Koch and Venkov \cite{KoVe} in 1975.

Since the generator rank \(d\) of \(G\) coincides with the \(p\)-class rank \(d_p(\mathrm{Cl}(K))\) of \(K\),
the inequality \(d<3\)  implies that the only imaginary quadratic fields \(K\) for which
the length \(\ell_p(K)\ge 2\) of their \(p\)-tower is an open problem ($p$ an odd prime), are those with \(d_p(\mathrm{Cl}(K))=2\).
In the case $p= 3$, such fields were investigated by Scholz and Taussky~\cite{SoTa}.
They proved that \(\ell_3(K)=2\) if the second \(3\)-class group \(\mathrm{G}_3^2(K)=\mathrm{Gal}(\mathrm{F}_3^2(K)\vert K)\) of \(K\)
is a metabelian \(3\)-group in Hall's isoclinism family \(\Phi_6\), and their proof was confirmed with different techniques by Heider and Schmithals \cite{HeSm} in 1982 and by Brink and Gold \cite{BrGo} in 1987.

No cases of \(p\)-towers of finite length \(\ell_p(K)>2\) ($p$ an odd prime) were known up to now and it is the main purpose of the present article to provide the first examples with \(\ell_3(K) = 3\). We note that examples of $2$-towers of length $3$ have appeared previously in~\cite{Bu1}. The main approach is to formulate conditions (see Theorem~\ref{main}) that guarantee that the Galois group
\(\mathrm{G}_3^3(K)=\mathrm{Gal}(\mathrm{F}_3^3(K)\vert K)\) has derived length \(3\).

The layout of the paper is as follows. In Section~2 we recall certain properties  shared by the Galois groups \(G = \mathrm{G}_p^\infty(K)=\mathrm{Gal}(\mathrm{F}_p^\infty(K)\vert K)\) when $K$ is an imaginary quadratic field. We also introduce the notions of transfer target type and transfer kernel type for the group $G$ and explain how these can be computed arithmetically. In Section~3 we recall how the $p$-group generation algorithm can be used to enumerate finite $p$-groups of fixed generator rank $d$. We also explain how some of the arithmetic data introduced in Section~2 can be used to constrain this enumeration. Finally, in Section~4 we formulate conditions on $G$ which are sufficient to cause an enumeration based search to terminate allowing us to identify the group $G$ as just one of two possible groups both of derived length $3$. This then yields a criterion for an imaginary quadratic field $K$ to have \(\ell_3(K) = 3\).


\section{Arithmetic restrictions on Galois groups}
As discussed in the introduction, the groups of interest in this paper are the pro-$p$ Galois groups
$G=\mathrm{G}_p^\infty(K)=\mathrm{Gal}(\mathrm{F}_p^\infty(K) / K)$ where $K$ is an imaginary quadratic field. Recall the following definition and notation.
\begin{definition}
The {\em derived series of a pro-$p$ group $G$} is defined recursively by $G^{(0)} = G$ and $G^{(n)} = [G^{(n-1)}, G^{(n-1)}]$ for $n \geq 1$ where $[H,K]$ denotes the closed subgroup generated by all commutators $[h,k] = h^{-1} k^{-1} h k$ with $h \in H$, $k \in K$. If $G^{(n-1)} \neq 1$ and $G^{(n)} = 1$ then we say that $G$ has {\em derived length $n$.}
\end{definition}
By class field theory, the abelianization $G^{ab} = G/ G^{(1)}$ is isomorphic to the $p$-class group $\mathrm{Cl}_p(K)$. In particular, it is a finite abelian $p$-group. This assertion also extends to each finite index subgroup $H$ of $G$ whose abelianization is isomorphic to the $p$-class group of the associated extension of $K$. Focusing on the subgroups in the derived series, one observes that these subgroups correspond to the fields in the $p$-class tower of $K$ and that this tower is finite exactly when $G$ is a finite $p$-group. In particular, the length of a finite tower corresponds to the derived length of $G$.

In~\cite{KoVe}  the notion of a Schur $\sigma$-group was introduced. We recall the definition below.
\begin{definition} A pro-$p$ group $G$ is called a {\em Schur $\sigma$-group} if the following conditions hold:
\begin{itemize}
\item[1.] The generator rank $d$ and relation rank $r$ of $G$ are equal.
\item[2.] $G^{ab}$ is finite.
\item[3.] There exists an automorphism $\sigma$ of order $2$ on $G$
which induces the inverse automorphism $x \mapsto x^{-1}$ on  $G^{ab}$. 
\end{itemize}
\end{definition}
After observing that $G=\mathrm{G}_p^\infty(K)$ is a Schur $\sigma$-group, one of the main results in \cite{KoVe} is that such a group must be infinite if $d \geq 3$. In particular, this means that the $p$-class tower of $K$ is infinite whenever the generator rank of $\mathrm{Cl}_p(K)$ is greater than or equal to $3$, and that $d = 1$ or $2$ must hold whenever $K$ has a finite tower. If $d= 1$, then $G$ is a finite cyclic group and so the associated tower has length $1$. We conclude that all finite towers of length greater than $1$ must have associated Galois group $G$ with $d = 2$.  In general, it is difficult to find finite towers of any appreciable length. For $p = 2$, examples  of length $3$ have been given in~\cite{Bu1}. In this paper, we give the first examples of length $3$ where $p = 3$ (the odd prime case).

The maximal subgroups in $G$ are normal of index $p$ and are thus in one-to-one correspondence with the maximal subgroups in $G / \Phi(G)$ where $\Phi(G) = G^p [G,G]$ is the Frattini subgroup. It follows that there are  $m = (p^d - 1)/(p - 1)$ such subgroups in $G$ and we denote them by $M_1, \ldots, M_m$. Each subgroup $M_i$ corresponds to an unramified cyclic extension of $K$ of degree $p$ which we denote by $K_i$. As noted earlier, we have $M_i^{ab} \cong \mathrm{Cl}_p(K_i)$. We now recall some facts about the transfer mapping and some terminology and notation introduced by the second author in earlier work~\cite{Ma1,Ma2}.

Given a group $G$, a subgroup $H$ of finite index $n$ and a left transversal $g_1, \ldots, g_n$ of $H$ in $G$, the transfer (or Verlagerung) map $V: G^{ab} \rightarrow H^{ab}$ is defined by $V(g G^{(1)}) = h_1 h_2 \ldots h_n H^{(1)}$ where $h_i \in H$ satisfies $g g_i = g_{\sigma(i)} h_i$ for some $\sigma \in S_n$. One can check that the map $V$ is well defined and does not depend on the choice of left transversal. The same map is also obtained if one uses right transversals.

\begin{definition}
Given a $d$-generated pro-$p$ group $G$ with maximal subgroups $M_1, \ldots, M_m$, the {\em transfer target type (TTT) of $G$} (denoted $\tau(G)$) is the sequence of abelianizations $M_1^{ab}, \ldots, M_m^{ab}$. The {\em transfer kernel type (TKT) of $G$} (denoted $\kappa(G)$) is the sequence $N_1, \ldots, N_m$ where $N_i$ is the kernel of the transfer map $V_i : G^{ab} \rightarrow M_i^{ab}$ for $1 \leq i \leq m$.
\end{definition}

\begin{remark}\label{equivalence}
The sequences appearing in the definitions of $\tau(G)$ and $\kappa(G)$ depend on the initial choice of an ordering on the subgroups $M_1, \ldots, M_m$. We will regard sequences obtained by making different choices as equivalent and will view $\tau(G)$ and $\kappa(G)$ as equivalence classes represented by the specified sequences.
\end{remark}

\begin{definition}
If $K$ is a number field and $p$ is some fixed prime then we define $\tau(K)$ to be the transfer target type of the Galois group $G = \mathrm{G}_p^\infty(K)$. We define $\kappa(K)$ similarly.
\end{definition}

\begin{remark}
As noted earlier, the components $M_i^{ab}$ of $\tau(K)$ can be evaluated by computing $\mathrm{Cl}_p(K_i)$ for the unramified cyclic extensions $K_i$ of degree $p$ over $K$. Class field theory also provides an arithmetic interpretation for the transfer kernels. The transfer map $V_i: G^{ab} \rightarrow M_i^{ab}$ corresponds to the extension homomorphism $\mathrm{Cl}_p(K) \rightarrow \mathrm{Cl}_p(K_i)$ and thus the kernel of $V_i$ can be viewed as the subgroup of ideal classes which become trivial when extended to the larger class group. (These classes are said to capitulate.)

\end{remark}

For the remainder of this section, we restrict attention to the situation where $p = 3$ and $K$ is an imaginary quadratic field with $\mathrm{Cl}_3(K) \cong [3,3]$ or, equivalently, $G^{ab} \cong [3,3]$. Here we are using the standard convention of listing the orders of cyclic factors in a direct sum decomposition to specify an abelian group. Since $d = 2$, there are four maximal subgroups  thus $\tau(G)$ and $\kappa(G)$ are sequences of length $4$. Observe that each component of the $TKT$ is a subgroup of $[3,3]$. There are five possible nontrivial subgroups. For ease of reference, the four subgroups of order $3$ (which correspond to the maximal subgroups in $G$) will be labeled $1, \ldots, 4$ and the whole group $[3,3]$ will be labeled $0$. We do not introduce a label for the trivial subgroup since it will not arise in our later considerations as a consequence of Hilbert's Theorem 94. Having fixed an ordering on the maximal subgroups, a $TKT$ can now be specified by listing the kernels of the corresponding transfer maps, in order, using these labels.

\begin{example}
The field $K = \mathbb{Q}(\sqrt{-9748})$ has $\mathrm{Cl}_3(K) \cong [3,3]$. After computing the $4$ unramified cyclic extensions $K_1, \ldots, K_4$ using Magma~\cite{MAGMA}, one can compute their $3$-class groups. These turn out to be: $[3,9]^3$, $[9,27]$ where the exponent indicates repeated occurrences of the same invariants. If one now computes the kernels of the extension homomorphisms $\mathrm{Cl}_3(K) \rightarrow \mathrm{Cl}_3(K_i)$ then one obtains the kernels of the transfer maps  $V_i : G^{ab} \rightarrow M_i^{ab}$ for $1 \leq i \leq 4$. For our chosen ordering of the subgroups, we obtained $\ker V_1 = M_1$, $\ker V_2 = M_4$, $\ker V_3 = M_3$ and $\ker V_4 = M_1$. Using the indices of the subgroups $M_i$ as the labels we can summarize this information by writing $\kappa(K) = (1,4,3,1)$. As noted in Remark~\ref{equivalence} above, this tuple depends on the chosen labeling and should be viewed as an equivalence class representative. As an illustration, if the subgroup $M_2$ was denoted $3$ and the subgroup $M_3$ was denoted $2$ then we would see that $(1, 4, 3, 1) \sim (1,2,4,1)$. If instead we swapped the labels on $M_1$ and $M_2$, then we would see that $(1,4,3,1) \sim (4, 2, 3, 2)$.
\end{example}

\section{The $p$-group generation algorithm}
The main result in the next section will be verified computationally using a backtrack search based on the $p$-group generation algorithm. In this section, we will give a brief overview of the algorithm and explain how the arithmetic restrictions discussed in the previous section can be used to constrain the search. The idea of using the $p$-group generation algorithm together with arithmetic information to search for certain Galois groups first appears in~\cite{BoLG}. Further examples can be found in~\cite{ BaBu,  BBH, BoNo, Bu1, No}.

\begin{definition}
The {\em lower exponent-$p$ central series of a finite $p$-group $G$} is defined recursively by $P_0(G) = G$ and $P_n(G) = P_{n-1}(G)^p [G, P_{n-1}(G)]$ for $n \geq 1$. If $P_{c}(G) = 1$ and $P_{c-1}(G) \neq 1$ then we say that $G$ has {\em $p$-class~$c$}.
\end{definition}

\begin{remark} If $G$ is a pro-$p$ group then we can extend the definitions by taking the closures of the corresponding subgroups. If $G$ is finitely generated then the subgroups in the lower exponent-$p$ central series are of finite index and so are also open subgroups. Thus a finitely generated pro-$p$ group has finite $p$-class if and only if it is a finite $p$-group.
\end{remark}

If $G$ has $p$-class $k$ and $1 \leq c \leq k$ then it is straightforward to show that the quotient $G/P_c(G)$ has $p$-class $c$. Moreover, this is the maximal possible quotient with this $p$-class.

\begin{definition}
If $G$ and $Q$ are finite $p$-groups and $G/P_c(G) \cong Q$ then we say that $G$ is a {\em descendant of $Q$} and that $Q$ is an {\em ancestor of $G$}. If $G$ has $p$-class $c+1$ then we say that it is an {\em immediate descendant of $Q$.} 
\end{definition}

Fixing $d$ we can visualize the finite $d$-generated $p$-groups  as arranged in a tree with the elementary abelian $p$-group of rank $d$ as the root and the groups of successively larger $p$-class in levels below. We connect two groups with an edge if one is an immediate descendant of the other. The $p$-group generation algorithm~\cite{Ob} allows one to compute all groups in  this tree  down to any desired $p$-class. The main idea behind the algorithm is that each finite $p$-group $G$ has only finitely many immediate descendants and these all arise as quotients of a certain uniquely determined covering group $G^*$. Moreover, there is an effective algorithm for computing this covering group and determining when two normal subgroups give rise to the same quotient. Given a presentation $G = F/R$, the covering group is defined $G^* = F/R^*$ where $R^* = [F,R] R^p$. Two quantities that play an important role in the algorithm are the $p$-multiplicator, which is the subgroup $R/R^*$ of $G^*$, and the nucleus, which is the subgroup $P_c(G^*) \cong P_c(F) R^*/ R^*$ where $c$ is the $p$-class of $G$.
For a more detailed description of the algorithm see \cite{Ob} and \cite[Chapter 9, pp. 353--372]{Ho}. 

The remaining results in this section explain how constraints on the structure of a group $G$ place constraints on the structure of its ancestors $G_c = G/P_c(G)$. The first lemma has been used explicitly or implicitly in a number of papers applying $p$-group generation to the computation of Galois groups.

\begin{lemma}\label{lem-epiab}
Let $f : G \rightarrow Q$ be a group epimorphism.
Let $N$ be a subgroup of finite index in $Q$ and $M = f^{-1}(N)$. Then $N^{ab}$ is a quotient of $M^{ab}$.
\end{lemma}
\begin{proof}
This is clear since the composition of the restriction of $f$ from $M$ to $f(M) = N$ and the natural epimorphism $N \rightarrow N^{ab}$ is a surjection from $M$ onto the abelian group $N^{ab}$ and must factor through $M^{ab}$.
\end{proof}
\begin{corollary}\label{cor-aqi}
Let $H$ be an ancestor of $G$. Then $H^{ab}$ is a quotient of $G^{ab}$ and the abelian groups appearing in $\tau(H)$ must be quotients of those appearing in $\tau(G)$.
\end{corollary}
As a simple example of how this result will be helpful, if we're searching for groups $G$ with $G^{ab}\cong [3,3]$ and we encounter a group $H$ with $H^{ab} \cong [3,9]$ or $[3,3,3]$ (or larger still) then we can eliminate $H$ and all of its descendants from our search. 

The next lemma and subsequent corollary are the new ingredients in this paper and play a crucial role in the final stages of the proof of our main result.
\begin{lemma}
Let $f : G \rightarrow Q$ be a group epimorphism.
Let $N$ be a subgroup of finite index in $Q$ and $M = f^{-1}(N)$. We have induced maps $\fbar: G^{ab} \rightarrow Q^{ab}$ and $\fbar: M^{ab} \rightarrow N^{ab}$. If we let $V_G: G^{ab} \rightarrow M^{ab}$ and $V_Q: Q^{ab} \rightarrow N^{ab}$ denote the transfer maps then
\[   V_Q \circ \fbar = \fbar \circ V_G. \]
\end{lemma}
\begin{proof}
Let $g_1, \ldots, g_n$ be a left transversal for $M$ in $G$ where $n = [Q : N] = [G : M]$ and define $q_i = f(g_i)$ for $1 \leq i \leq n$. Then $q_1, \ldots, q_n$ form a left transversal for $N$ in $Q$. 

Let $g \in G$ and suppose that $g g_i = g_{\sigma(i)} m_i$ for some $\sigma \in S_n$ and $m_i \in M$. From the definition of $V_G$, we have 
\[ (\fbar \circ V_G)(g G^{(1)}) = \fbar(m_1 m_2 \ldots m_n M^{(1)}) = f(m_1) \ldots f(m_n) N^{(1)}.
\]
On the other hand, we have
\[  (V_Q \circ \fbar)(g G^{(1)}) =   V_Q( f(g) Q^{(1)} ) = f(m_1) \ldots f(m_n) N^{(1)}
\]
since $f(g) q_i = f(g g_i) = f(g_{\sigma(i)} m_i) = q_{\sigma(i)} f(m_i)$ for $1 \leq i \leq n$. It follows that 
$V_Q \circ \fbar = \fbar \circ V_G$.
\end{proof}
\begin{corollary}\label{cor-tkt}
Suppose that, in addition to the conditions specified in the preceding lemma, the map  $\fbar: G^{ab} \rightarrow Q^{ab}$ is an isomorphism. Then $\ker V_G \subseteq \ker V_Q$ (using the isomorphism to view these as subgroups of the same group)  and so each subgroup appearing in $\kappa(G)$ must be contained in the corresponding subgroup in $\kappa(Q)$.
\end{corollary}
\begin{proof} This follows since
\[ \ker V_G = V_G^{-1}(1) \subseteq (\fbar \circ V_G)^{-1}(1) = \fbar^{-1} (\ker V_Q).
\]
\end{proof}

Next we have a statement about $\sigma$-automorphisms.
\begin{lemma}\label{lem-sigma}
Let $G$ be a pro-$p$ group and let $f : G \rightarrow G_c$ be the natural epimorphism.
If $G$ possesses an automorphism $\sigma$ which has order $2$ and which restricts to the inversion automorphism on $G^{ab}$ then so does $G_c$.
\end{lemma}
\begin{proof}
The subgroup $P_c(G)$ is characteristic for all $c$ so $\sigma$ has a well defined restriction to $G_c$ whose order must be either $1$ or $2$. Observe that $\sigma$ can be restricted further to $G_c^{ab}$ and that the restriction map on automorphisms passes through $G^{ab}$. Since the restriction of $\sigma$ to $G^{ab}$ is the inversion map, the same statement holds for $G_c^{ab}$ and so the restriction has order $2$.
\end{proof}

We also need the following results. The first is proved in~\cite[Proposition 2]{BoNo}. The second appears in~\cite[Theorem 4.4]{No}. We outline the proof of the second result for the reader's convenience.
\begin{proposition}\label{prop-rel-bd}
Let $G$ be a pro-$p$ group with finite abelianization. For all $c \geq 1$ the difference between the ranks of the $p$-multiplicator and the nucleus of $G_c$ is at most $r(G) = \dim H^2(G, \mathbb{F}_p)$.
\end{proposition}
\begin{theorem}\label{thm-nov}
Let $G$ be a pro-$p$ group, $N$ a finite index normal subgroup, and $V$ a word, and assume $P_c(G) \leq N$. If
$[G_c: V(N/P_c(G))] = [G_{c+1}: V(N/P_{c+1}(G))]$ then 
\[ \frac{N/P_c(G)} {V(N/P_c(G))} \cong N/V(N). \]
It follows that
\[ \frac{N/P_{c'}(G)} {V(N/P_{c'}(G))} \cong N/V(N) \]
for all $c' \geq c$.
\end{theorem}
\begin{proof}
We first remark that $V(G)$ denotes the verbal subgroup of $G$ generated using the word $V$.  Recall that verbal subgroups are characteristic and that if $f: G \rightarrow Q$ is an epimorphism then $f(V(G)) = V(Q)$.

Since $V(N)$ is characteristic in $N$ which is normal in $G$ we can form the quotient $G/V(N)$. It is straightforward to verify that
\[  \frac{G/V(N)}{P_c(G/V(N))} \cong  \frac{G}{P_c(G) V(N)} \cong \frac{G_c}{V(N / P_c(G))}
\]
for every $c \geq 1$. The hypotheses then imply that 
\[ \frac{G/V(N)}{P_c(G/V(N))} \cong \frac{G/V(N)}{P_{c+1}(G/V(N))}.
\]
This can only happen if $P_c(G/V(N)) = 1$ or, equivalently, $P_c(G) \leq V(N)$. But then $V(N/P_c(G)) = V(N)/P_c(G)$ which leads to an isomorphism 
\[ \frac{N/P_{c}(G)} {V(N/P_{c}(G))} \cong N/V(N). \]
This isomorphism factors through the natural epimorphism from $N/V(N)$ to $\frac{N/P_{c'}(G)} {V(N/P_{c'}(G))}$ for $c' \geq c$ yielding the final statement in the theorem.
\end{proof}
If $V$ is the commutator word $[x,y] = x^{-1} y^{-1} x y$ then $V(G) = [G,G] = G^{(1)}$ and we have the following corollary.
\begin{corollary}\label{cor-aqi-stab}
Let $G$ be a pro-$p$ group, $N$ a finite index normal subgroup with $P_c(G) \leq N$ and let $N_c = N/P_c(G) \leq G_c$. If
$N_c^{ab} \cong N_{c+1}^{ab}$ then $N_{c'}^{ab} \cong N^{ab}$ for all $c' \geq c$.
\end{corollary}
\begin{proof}
If $V(G) = G^{(1)}$ then the condition in Theorem~\ref{thm-nov} is $[G_c: N_c^{(1)}] = [G_{c+1}: N_{c+1}^{(1)}]$. Using Lemma~\ref{lem-epiab}, one can see that this is equivalent to $N_c^{ab} \cong N_{c+1}^{ab}$.
\end{proof}

\section {Criterion for a tower of length $3$}
We can now state our main result.
\begin{theorem}\label{main}
Let $G$ be a Schur $\sigma$-group satisfying:
\begin{itemize}
\item[(i)] $G^{ab} \cong [3,3]$
\item[(ii)] $\tau(G) = \{ [3,9]^3, [9,27] \}$ and
\item[(iii)] $\kappa(G)$ equivalent to $(1,4,3,1)$.
\end{itemize}
Then, up to isomorphism,  $G$ is one of two possible finite $3$-groups of order $3^8$. Both of these groups have derived length $3$.
\end{theorem}
\begin{proof}
$G$ is $2$-generated so every finite quotient $G_c = G/P_c(G)$ is a descendant of $G_1 = G/P_1(G) \cong [3,3]$. At the outset, we do not know if there are infinite pro-$3$ groups or finite $3$-groups of arbitrarily large $3$-class satisfying the given conditions. If such groups $G$ having $3$-class at least $c$ exist then the quotient $G_c$ in each case will have $3$-class exactly $c$. We will proceed by finding all candidates for $G_c$ with $3$-class exactly $c$, for larger and larger values of $c$. This process will terminate allowing us to deduce that $G$ is finite with bounded $3$-class. We have used Magma~\cite{MAGMA} to carry out the needed descendant computations and to check various conditions.

First, $G^{ab} \cong [3,3]$ so by Corollary~\ref{cor-aqi} we can rule out all $3$-groups and their descendants whose abelianization is larger than $[3,3]$. Of the $7$ immediate descendants of $G_1 \cong [3,3]$, only two satisfy this restriction. Of these, only one possesses a $\sigma$-automorphism so by Lemma~\ref{lem-sigma} we can eliminate the other. We conclude that the unique group remaining must be $G_2$ up to isomorphism.

To find $G_3$ we need only consider the immediate descendants of $G_2$. There are $11$ such groups. Many of these have maximal subgroups of index $3$ with abelianization $[3,3,3]$. It follows that these groups and their descendants can be eliminated by Corollary~\ref{cor-aqi}. This leaves only five possible groups of $p$-class $3$. Three more of these can be eliminated because they fail to satisfy the bound in Proposition~\ref{prop-rel-bd} leaving only two possibilities for $G_3$. For both of these groups the TTT is $[3,9]^4$. Furthermore, one of these groups has two descendants both with exactly the same TTT. By Corollary~\ref{cor-aqi-stab}, these groups and all of their descendants have the same TTT and can be eliminated. This leaves only one possible candidate group for $G_3$. 

Our candidate for $G_3$ has $4$ immediate descendants only one of which has a Schur $\sigma$-automorphism so by Lemma~\ref{lem-sigma} this must be $G_4$. This in turn has $14$ immediate descendants all of which possess a $\sigma$-automorphism. However, only six of them satisfy the bound given in  Proposition~\ref{prop-rel-bd}. These remaining six groups of $3$-class $5$ all have order $3^8$ and TTT $ = \{ [3,9]^3, [9,27] \}$. Computing the TKTs for these groups, three of them (and their descendants) can be eliminated  using Corollary~\ref{cor-tkt}. Two of the groups that remain have TKT equivalent to $(1,4,3,1)$. The third has TKT $(0,2,3,1)$. This latter group (and its descendants) cannot be immediately ruled out since it might be possible that the subgroup represented by the first component becomes smaller in a descendant. In particular, the TKT $(2,2,3,1)$ or $(3,2,3,1)$ might appear both of which are equivalent to $(1, 4, 3, 1)$. Examining the immediate descendants more closely though, one sees that there are four of them and that they all have TTT $= \{ [3,9]^3, [27,27] \}$. This means that these groups and their descendants can be eliminated using Corollary~\ref{cor-aqi}. Thus there are only two candidates for $G_5$.

Since both of the candidates for $G_5$ do not have any descendants (they are said to be terminal), we conclude that there do not exist any finite groups of $3$-class $6$ satisfying all of the properties necessary for them to be $G_6$. It follows that $G$ is finite with $3$-class at most $5$. We can also see that $G \ncong G_1, G_2, G_3, G_4$ by comparing $TTT$s and/or observing that these groups all have relation rank $> 2$. It follows that $G \cong G_5$ and is one of the two candidate groups found above. Both groups have derived length $3$.

One of these groups has the following power-commutator presentation
\begin{eqnarray*}
G = \langle g_1, g_2, g_3, g_4, g_5, g_6, g_7, g_8  &\mid&
g_1^3 = g_5 g_6^2, \quad
g_2^3 = g_4 g_8, \quad
g_3^3 = g_6^2 g_7 g_8, \quad
g_4^3 = g_7^2, \quad
g_5^3 = g_8^2, \\
&&
\lbrack g_2, g_1 \rbrack = g_3, \quad
\lbrack g_3, g_1 \rbrack = g_4, \quad
\lbrack g_3, g_2 \rbrack = g_5, \quad 
\lbrack g_5, g_1 \rbrack = g_7, \\
&&
\lbrack g_5, g_2 \rbrack = g_6, \quad 
\lbrack g_5, g_3 \rbrack = g_7, \quad 
\lbrack g_6, g_1 \rbrack = g_7, \quad
\lbrack g_6, g_2 \rbrack = g_8
\: \rangle.
\end{eqnarray*}
Note that  relations of the form $g_i^3 = 1$ and $\lbrack g_i, g_j \rbrack = 1$ are also present but have not been displayed as is the usual convention. A presentation for the other group can be obtained by replacing the first power relation $g_1^3 = g_5 g_6^2$ with $g_1^3 = g_5 g_6^2 g_8^2$. Presentations for the $p$-quotients $G_c$ that appear in the intermediate stages of the computation can be obtained from the above presentations. The smallgroup database identifiers for these quotients are: $G_1 \cong (9,2)$, $G_2 \cong (27,3)$, $G_3 \cong (243,8)$ and $G_4 \cong (729,54)$.
\end{proof}

Theorem~\ref{main} leads immediately to the following  sufficient criterion for an imaginary quadratic field to have a $3$-class tower of length $3$.
\begin{corollary}
If $K$ is an imaginary quadratic field such that $G=\mathrm{G}_3^\infty(K)$ satisfies the conditions
in Theorem~\ref{main} then $K$ has $3$-class tower of length exactly $3$. The field $K = \mathbb{Q}(\sqrt{-9748})$ is one such example.
\end{corollary}

\begin{remark}
In~\cite[p. 41]{SoTa}, Scholz and Taussky claim that $K = \mathbb{Q}(\sqrt{-9748})$ has $3$-class tower of length $2$ which does not agree with the results of our work. We are not the first to observe that there may be a problem with this claim. It is shown in~\cite{BrGo} that there is no group theoretic restriction resulting from the capitulation conditions which prevents the length from being greater than $2$. We believe that we are the first to establish that the length is exactly $3$ under the conditions given above.
\end{remark}

\section{Acknowledgements}
Research of the last author is supported by the Austrian Science Fund (FWF): P 26008-N25. Both authors would like to thank Nigel Boston and Farshid Hajir for helpful discussions.

\end{document}